\documentclass[11pt]{amsart}
\usepackage{amsmath, amssymb, verbatim,color, amscd}
\newcommand{\de}{\partial}
\newcommand{\db}{\overline{\partial}}

\newcommand{\ddbar}{\sqrt{-1} \partial \overline{\partial}}
\newcommand{\Ric}{\mathrm{Ric}}
\newcommand{\ov}[1]{\overline{#1}}
\newcommand{\mn}{\sqrt{-1}}
\newcommand{\tr}[2]{\mathrm{tr}_{#1}{#2}}
\newcommand{\ti}[1]{\tilde{#1}}
\newcommand{\vp}{\varphi}

\newcommand{\ve}{\varepsilon}

\renewcommand{\leq}{\leqslant}
\renewcommand{\geq}{\geqslant}

\numberwithin{equation}{section}

\begin{document}
\newtheorem{claim}{Claim}
\newtheorem{theorem}{Theorem}[section]
\newtheorem{lemma}[theorem]{Lemma}
\newtheorem{corollary}[theorem]{Corollary}
\newtheorem{proposition}[theorem]{Proposition}
\newtheorem{question}[theorem]{Question}
\newtheorem{conjecture}[theorem]{Conjecture}

\theoremstyle{definition}
\newtheorem{remark}[theorem]{Remark}

\title[Infinite time singularities of the K\"ahler-Ricci flow]{Infinite time singularities of the K\"ahler-Ricci flow}
\author[V. Tosatti]{Valentino Tosatti}
\thanks{The first-named author is supported in part by NSF grant DMS-1308988 and by a Sloan Research Fellowship. The second-named author is supported in part by grant NSFC-11271015.}
\address{Department of Mathematics, Northwestern University, 2033 Sheridan Road, Evanston, IL 60208}
\email{tosatti@math.northwestern.edu}
  \author[Y. Zhang]{Yuguang Zhang}
\address{Mathematical Sciences Center,  Tsinghua University,  Beijing 100084, P.R.China.}
\email{yuguangzhang76@yahoo.com}

\begin{abstract} We study the long-time behavior of the K\"ahler-Ricci flow on compact K\"ahler manifolds. We give an almost complete classification of the singularity type of the flow at infinity, depending only on the underlying complex structure.
If the manifold is of intermediate Kodaira dimension and has semiample canonical bundle, so that it is fibered by Calabi-Yau varieties, we show that parabolic rescalings around any point on a smooth fiber converge smoothly to a unique limit, which is the product of a Ricci-flat metric on the fiber and of a flat metric on Euclidean space. An analogous result holds for collapsing limits of Ricci-flat K\"ahler metrics.

\end{abstract}

\maketitle

\section{Introduction}

This paper has two main goals. The first goal is to prove some results about collapsing limits of either Ricci-flat K\"ahler metrics or certain solutions of the K\"ahler-Ricci flow, which complement and complete the very recent results proved by Weinkove, Yang and the first-named author in \cite{TWY}. In a nutshell, we prove higher-order a priori estimates for the certain rescalings of the solutions restricted to the fibers of a holomorphic map, and we then classify the possible blowup limit spaces. As it turns out, these limits are unique, and are given by the product of a Ricci-flat K\"ahler metric on a compact Calabi-Yau manifold times a Euclidean metric on $\mathbb{C}^m$.

The second goal of the paper is to study infinite time singularities of the K\"ahler-Ricci flow. As we will recall below, these are divided into two classes called ``type IIb'' and ``type III'' (the other cases ``type I'' and ``type IIa'' are reserved for finite-time singularities). Assuming that the canonical bundle is semiample (which conjecturally is always true whenever we have a long-time solution), we give an almost complete classification (which is complete in complex dimension $2$) of which singularity type arises, depending on the complex structure of the underlying manifold. As a consequence, in all the cases that we cover, the singularity type does not depend on the initial K\"ahler metric.

We now discuss the first goal in detail. As we said above, we will consider two different setups. In the first setup, which is the same as in \cite{To, GTZ, GTZ2, HT, TWY}, we
let $(X,\omega_X)$ be a compact Ricci-flat Calabi-Yau $(n+m)$-manifold, $n>0$, which admits a holomorphic map $\pi:X\to Z$ where $(Z,\omega_Z)$ is a
compact K\"ahler manifold. Denote by $B=\pi(X)$ the image of $\pi$, and assume that $B$ is an irreducible normal subvariety of $Z$ with
dimension $m>0$, and that the map $\pi:X\to B$ has connected fibers. Denote by $\chi=\pi^*\omega_Z$, which is a smooth
nonnegative real $(1,1)$-form on $X$ whose cohomology class lies on the boundary of the K\"ahler cone of $X$, and denote also by
$\chi$ the restriction of $\omega_Z$ to the regular part of $B$. In practice we can take $Z=\mathbb{CP}^N$ if $B$ is a projective variety, or $Z=B$ if $B$ is smooth.
In general, given a map $\pi:X\to B$ as above, there is a proper analytic subvariety $S'\subset B$, which consists of the singular points of $B$ together with the critical values of $\pi$, such that if we let $S:=\pi^{-1}(S')$ we obtain an analytic subvariety of $X$ and $\pi:X\backslash S\to B\backslash S'$ is a smooth submersion. For any $y\in B\backslash S'$ the fiber
$X_y=\pi^{-1}(y)$ is a smooth Calabi-Yau manifold of dimension $n$, and it is equipped with the K\"ahler metric $\omega_X|_{X_y}$.
Consider the K\"ahler metrics on $X$ given by $\ti{\omega}=\ti{\omega}(t)=\chi+e^{-t}\omega_X$, with $t\geq 0$, and call
$\omega=\omega(t)=\ti{\omega}+\ddbar\vp$ the unique Ricci-flat K\"ahler metric on $X$ cohomologous to $\ti{\omega}$, which exist thanks to Yau's Theorem \cite{Ya}, with potentials normalized
by $\sup_X\vp=0$. They satisfy a family of complex Monge-Amp\`ere equations
\begin{equation}\label{ma}
\omega^{n+m}=(\ti{\omega}+\ddbar\vp)^{n+m}=c_t e^{-nt} \omega_X^{n+m},
\end{equation}
where $c_t$ is a constant that has a positive limit as $t\to \infty$. Thanks to the works \cite{To, ST, GTZ, GTZ2, TWY, Z} we know that in this case the Ricci-flat metrics $\omega(t)$ have bounded diameter and collapse locally uniformly on $X\backslash S$ to a canonical K\"ahler metric on $B\backslash S'$, and when the fibers $X_y$ are tori then the collapsing is smooth and with locally bounded curvature on $X\backslash S$. Also, for general smooth fibers, the rescaled metrics along the fibers $e^t\omega|_{X_y}$ converge in $C^\alpha$ to a Ricci-flat metric on $X_y$.

The second setup is as follows (cf. \cite{ST,ST2,ST3, FZ,TWY}). Now $(X,\omega_X)$ is a compact K\"ahler $(n+m)$-manifold, $n>0$, with semiample canonical bundle
and Kodaira dimension equal to $m>0$. As explained for example in \cite{TWY}, sections of $K_X^\ell,$ for $\ell$ large, give rise to a fiber space $\pi:X\to B$ called the {\em Iitaka fibration} of $X$, with $B$ a normal projective variety of dimension $m$ and the smooth fibers $X_y=\pi^{-1}(y),y\in B\backslash S'$ also Calabi-Yau $n$-manifolds. We let $\chi$ be the restriction of $\frac{1}{\ell}\omega_{FS}$ to $B$, as well as its pullback to $X$. This time we consider the solution $\omega=\omega(t)$ of the {\em normalized} K\"ahler-Ricci flow
\begin{equation}\label{nkrf}
\frac{\de}{\de t}\omega=-\Ric(\omega)-\omega,\quad\omega(0)=\omega_X,
\end{equation}
which exists for all $t\geq 0$. Thanks to \cite{ST, ST2, ST3, FZ, TWY} we have that the evolving metrics have uniformly bounded scalar curvature and collapse locally uniformly on $X\backslash S$ to a canonical K\"ahler metric on $B\backslash S'$, and again we have smooth collapsing when the smooth fibers are tori. Again, the rescaled metrics along the fibers $e^t\omega|_{X_y}$ converge in $C^\alpha$ to a Ricci-flat metric on $X_y$.

Our first goal is to improve this last statement by showing the following higher-order estimates:

\begin{theorem}\label{main1} Assume we are in either the first or second setup. Given a compact subset $K\subset B\backslash S'$ and $k\geq 0$ there is a constant $C_k$ such that
\begin{equation}\label{bounds}
\left\| e^t\omega|_{X_y}\right\|_{C^k(X_y,\omega_X|_{X_y})}\leq C_k,\quad
e^t\omega|_{X_y}\geq C_0^{-1} \omega_X|_{X_y},
\end{equation}
holds for all $t\geq 0$ and for all $y\in K$.
\end{theorem}
The case when $k=1$ of Theorem \ref{main1} was proved recently in \cite{TWY}, who also showed that as $t\to \infty$ the metrics $e^t\omega|_{X_y}$ converge in $C^\alpha(X_y,\omega_X|_{X_y}),$ $0<\alpha<1$, to the unique Ricci-flat K\"ahler metric on $X_y$ cohomologous to $\omega_X|_{X_y}$. Combining this and Theorem \ref{main1} we immediately conclude:

\begin{corollary}\label{main2} Assume we are in either the first or second setup. Given $y\in B\backslash S'$ we have
\begin{equation}\label{rescaled}
e^t\omega|_{X_y}\to \omega_{SRF,y},
\end{equation}
in the smooth topology on $X_y$ as $t\to\infty$, where $\omega_{SRF,y}$ is the unique Ricci-flat metric on $X_y$
cohomologous to $\omega_X|_{X_y}$.
\end{corollary}

This solves affirmatively a problem raised by the first-named author \cite[Question 4.1]{To3}, \cite[Question 3]{To4}.

In the second setup we investigate in more detail the nature of the singularity of the K\"ahler-Ricci flow as $t\to\infty$, and more precisely we want to determine
the possible limits (or ``tangent flows'') that one obtains by parabolically rescaling the flow around the ``singularity at infinity''.
In general determining whether tangents to solutions of geometric PDEs are unique or not is a very challenging problem. In the setup as above, we prove that such parabolic rescalings (centered at a point on a smooth fiber) converge to a unique limit. More precisely we have:

\begin{theorem}\label{main3}
Assume the same setup for the K\"ahler-Ricci flow \eqref{nkrf} as above.
Given $t_k\to\infty$ and $x \in X\backslash S$, let $V$ be the preimage of a sufficiently small neighborhood of $\pi(x)$ and let
$$\omega_k(t)=e^{t_k} \omega(te^{-t_k} + t_k)$$
be the parabolically rescaled flows, then after passing to a subsequence the flows $(V, \omega_k(t), x), t\in[-1,0],$ converge in the smooth Cheeger-Gromov sense to $(X_y\times\mathbb{C}^m,\omega_\infty,(x,0))$, where $y=\pi(x)$,
and $\omega_\infty$ is the product of the unique Ricci-flat K\"ahler metric on $X_y$ in the class $[\omega_X|_{X_y}]$ and of a flat metric on $\mathbb{C}^m$, viewed as a static solution of the flow. In particular, there is a unique such limit up to holomorphic isometry.
\end{theorem}

In fact, analogous results as in Theorem \ref{main3} also hold in the first setup of collapsing of Calabi-Yau manifolds, with essentially the same proofs. We briefly discuss this at the end of Section \ref{blowup}.

We now come to the second goal of this paper, which is to characterize infinite time singularities of the K\"ahler-Ricci flow.
Recall that a long-time solution of the {\em unnormalized} K\"ahler-Ricci flow
\begin{equation}\label{krf}
\frac{\de}{\de t}\omega=-\Ric(\omega),\quad\omega(0)=\omega_X,
\end{equation}
is called type IIb if
$$\sup_{X\times [0,\infty)} t|\mathrm{Rm}(\omega(t))|_{\omega(t)}=+\infty,$$
and type III if
$$\sup_{X\times [0,\infty)} t|\mathrm{Rm}(\omega(t))|_{\omega(t)}<+\infty.$$
If the flow is instead normalized as in \eqref{nkrf}, then one has to remove the factor of $t$ from these conditions.

When $\dim X=1$, so $X$ is a compact Riemann surface, it follows easily from the work of Hamilton \cite{Ha} that all long-time solutions of the flow are of type III, and these are exactly the K\"ahler-Ricci flow solutions on compact Riemann surfaces of genus $g\geq 1$. Recent work of Bamler \cite{Ba} shows that the same statement holds for the Ricci flow (not K\"ahler) on compact Riemannian $3$-manifolds. Homogeneous type-III solutions were studied by Lott \cite{Lo}. The simplest example of a type IIb solution on a compact Riemannian $4$-manifold is a non-flat Ricci-flat K\"ahler metric on a $K3$ surface, which exists thanks to Yau \cite{Ya}, which provides a static solution of the K\"ahler-Ricci flow which is of type IIb. As we will see below, there are also non-static type IIb solutions in dimension $4$.
In the case of higher dimensional K\"ahler-Ricci flows the only general result that we are aware of is due to Fong-Zhang \cite{FZ} who proved that if $\pi:X\to B$ is a holomorphic submersion with fibers equal to complex tori, with $c_1(B)<0$, with $X$ projective and initial K\"ahler class rational, then the flow is of type III. This used key ideas of Gross-Tosatti-Zhang \cite{GTZ} in the case of collapsing of Ricci-flat metrics, and the projectivity/rationality assumptions were recently removed in \cite{HT}. See also \cite{Gi} for the case when $X$ is a product.

Our goal is to have a more detailed understanding of which long-time solutions of the K\"ahler-Ricci flows are of type IIb or type III.
Our main tool is the following observation.
\begin{proposition}\label{crit}
Let $X$ be a compact K\"ahler manifold with $K_X$ nef and which contains a possibly singular rational curve $C\subset X$ (i.e. $C$ is the birational image of a nonconstant holomorphic map $f:\mathbb{CP}^1\to X$) such that
$\int_C c_1(X)=0$. Then any solution of the unnormalized K\"ahler-Ricci flow \eqref{krf} on $X$ must be of type IIb.
\end{proposition}
For example if $X$ is a minimal K\"ahler surface of general type which contains a $(-2)$-curve, and there are many such examples, then this result applies and we get examples of non-static long-time K\"ahler-Ricci flow solutions in real $4$ dimensions which are of type IIb.

Now in general if a K\"ahler-Ricci flow solution on a compact K\"ahler manifold $X$ exists for all $t\geq 0$, then necessarily the canonical bundle $K_X$ is nef (the converse is also true \cite{TZ}). The abundance conjecture in algebraic geometry (or rather its generalization to K\"ahler manifolds) predicts that if $K_X$ is nef then in fact it is semiample, namely $K_X^\ell$ is base-point-free for some $\ell\geq 1$. We will assume that this is the case, and so the sections of $K_X^\ell$, some $\ell$ large, define a fiber space $\pi:X\to B$, the Iitaka fibration of $X$. As before we denote by $S'\subset B$ the singular set of $B$ together with the critical values of $\pi$, and let $S=\pi^{-1}(S')$, so that $S=\emptyset$ precisely when $B$ is smooth and $\pi$ is a submersion. Somewhat imprecisely, we will refer to $S$ as the set of singular fibers of $\pi$. We will also write $X_y=\pi^{-1}(y)$ for one of the smooth fibers, $y\in B\backslash S'$.
We then have the following result.
\begin{theorem}\label{main4}
Let $X$ be a compact K\"ahler $n$-manifold with $K_X$ semiample, and consider a solution of the K\"ahler-Ricci flow \eqref{krf}.
\begin{itemize}
\item $\kappa(X)=0$
\begin{itemize}
\item[$\diamond$] $X$ is not a finite quotient of a torus $\Rightarrow$ Type IIb
\item[$\diamond$] $X$ is a finite quotient of a torus $\Rightarrow$ Type III
\end{itemize}
\item $\kappa(X)=n$
\begin{itemize}
\item[$\diamond$] $K_X$ is ample $\Rightarrow$ Type III
\item[$\diamond$] $K_X$ is not ample $\Rightarrow$ Type IIb
\end{itemize}
\item $0<\kappa(X)<n$
\begin{itemize}
\item[$\diamond$] $X_y$ is not a finite quotient of a torus $\Rightarrow$ Type IIb
\item[$\diamond$] $X_y$ is a finite quotient of a torus and $S=\emptyset$ $\Rightarrow$ Type III
\end{itemize}
\end{itemize}
In particular, in these cases the type of singularity does not depend on the initial metric.
\end{theorem}

This result leaves out only the case when $0<\kappa(X)<n$ and the general fiber $X_y$ is a finite quotient of a torus and there are singular fibers. In this case, sometimes one can find a component of a singular fiber which is uniruled, in which case the flow is type IIb by Proposition \ref{crit}, but in some other cases there is no such component, and then it seems highly nontrivial to determine whether the flow is of type III or IIb. In any case, we expect the singularity type to be always independent of the initial metric.

We are able to solve this question when $n=2$, and complete the singularity classification (note that since the abundance conjecture holds for surfaces, it is enough to assume that $K_X$ is nef). In this case
$\pi:X\to B$ is an elliptic fibration with $X$ a minimal properly elliptic K\"ahler surface, and with some singular fibers. Recall that in Kodaira's terminology (cf. \cite{bhpv}) a fiber of type $mI_0,m>1$ is a smooth elliptic curve with multiplicity $m$. We can then complete Theorem \ref{main4} in dimension $2$ as follows:
\begin{theorem}\label{surfaces}
Let $X$ be a minimal K\"ahler surface with $\kappa(X)=1$, and let $\pi:X\to B$ be an elliptic fibration which is not a submersion everywhere. Then the flow is type III if and only if the only singular fibers of $\pi$ are of type $mI_0,m>1$.
\end{theorem}

In the Minimal Model Program, we sometime  have many different  minimal models in one birational equivalence class. In \cite{Ka2}, Kawamata shows that different minimal models can be connected by a sequence of flops.   A corollary of Proposition \ref{crit} is the relationship of this multi-minimal model phenomenon  and the singularity  type of the K\"ahler-Ricci flow.

\begin{corollary}\label{corollary}
Let  $X$ be a projective  $n$-manifold with $K_X$ nef.  If $X$ has a different minimal model $Y$, i.e. a $\mathbb{Q}$-factorial terminal variety $Y$ with $K_Y$ nef and with a birational map $\alpha: Y \dashrightarrow X$, which is not isomorphic to $X$, then
 any solution of the K\"ahler-Ricci flow \eqref{krf} on $X$ is of type IIb.  As a consequence, if there is a flow solution of type III, then $X$ is the unique minimal model in its birational equivalence class.
\end{corollary}

This paper is organized as follows. In Section \ref{estimates} we prove Theorem \ref{main1}. Theorem \ref{main3} is proved in Section \ref{blowup}, and finally in Section \ref{singular} we give the proofs of Proposition \ref{crit}, Theorems \ref{main4} and \ref{surfaces} and Corollary \ref{corollary}.\\

\noindent {\bf Acknowledgements:} We are grateful to M. Gross, H.-J. Hein, B. Weinkove for very useful discussions.
Most of this work was carried out while the first-named author was visiting the Mathematical Science Center of Tsinghua University in Beijing, which he would like to thank for the hospitality.

\section{Estimates along the fibers}\label{estimates}

In this section we prove Theorem \ref{main1} by deriving {\em a priori} $C^\infty$ estimates for the rescaled metrics restricted to a smooth fiber.

We first consider the setup of collapsing of Ricci-flat K\"ahler metrics, as described in the Introduction.
We need two preliminary results. The first one follows from work of the first-named author \cite{To} (see also \cite[Lemma 4.1]{GTZ}).
\begin{lemma}[see \cite{To, GTZ}]\label{c2es} Given any compact set $K\subset X\backslash S$ there is a constant $C$ such that on $K$ the Ricci--flat metrics $\omega$ satisfy
\begin{equation}\label{c2}
C^{-1}(\chi+e^{-t}\omega_X)\leq \omega\leq C(\chi+e^{-t}\omega_X),
\end{equation}
for all $t>0$.
\end{lemma}

The second ingredient is the following local estimate for Ricci-flat K\"ahler metrics, which is contained in \cite[Sections 3.2 and 3.3]{HT} and is an adaptation (and in fact a special case) of a similar result from \cite{SW} for the K\"ahler-Ricci flow.
\begin{lemma}[see \cite{HT}]\label{localest}
Let $B_1(0)$ be the unit ball in $\mathbb{C}^{n}$ and let $\omega_E$ be the Euclidean metric.
Assume that $\omega$ is a Ricci-flat K\"ahler metric on $B_1(0)$ which satisfies
\begin{equation}\label{2assum}
A^{-1}\omega_{E}\leq \omega \leq A\, \omega_{E},
\end{equation}
for some positive constant $A$.
Then for any $k\geq 1$ there is a constant $C_k$ that depends only on $k,n,m,A$ such that on $B_{1/2}(0)$ we have
\begin{equation}\label{toprove}
\|\omega\|_{C^k(B_{1/2}(0),\omega_E)}\leq C_k.
\end{equation}
\end{lemma}

\begin{proof}[Proof of Theorem \ref{main1} for Ricci-flat K\"ahler metrics]
Fix a point $y_0\in B\backslash S'$, a point $x\in X_{y_0}$ and a small chart $U\subset X$ with coordinates
$(y,z)=(y_1,\dots,y_m,z_1,\dots,z_{n})$ centered at $x$, where $(y_1,\dots,y_m)$ are the pullback
via $\pi$ of coordinates on the image $\pi(U)\subset B$ and such that the projection $\pi$ in this coordinates is just $(y,z)\mapsto y$. Such a chart exists because $\pi$ is a submersion near $x$ (see \cite[p.60]{Ko}).
We can assume that $U$ equals the polydisc where $|y_i|<1$ for $1\leq i\leq m$ and $|z_\alpha|<1$ for $1\leq \alpha\leq n$.

For each $t\geq 0$, consider the polydiscs $B_t=B_{e^{t/2}}(0)\subset\mathbb{C}^m$, let $D$ be the unit polydisc in $\mathbb{C}^{n}$ and define maps
$$F_t:B_t\times D\to U, \quad F_t(y,z)=(ye^{-t/2},z).$$
Note that the stretching $F_t$ is the identity when restricted to $\{0\}\times D$, and as $t$ approaches zero the polydiscs $B_t\times D$ exhaust $\mathbb{C}^m\times D$.
On $U$ we can write
\[\begin{split}
\omega_X(y,z)&=\mn\sum_{i,j=1}^m g_{i\ov{j}}(y,z) dy_i\wedge d\ov{y}_j+2\mathrm{Re}
\mn\sum_{i=1}^m\sum_{\alpha=1}^{n} g_{i\ov{\alpha}}(y,z) dy_i\wedge d\ov{z}_\alpha\\
&+\mn\sum_{\alpha,\beta=1}^{n}g_{\alpha\ov{\beta}}(y,z) dz_{\alpha}\wedge d\ov{z}_{\beta},
\end{split}\]
so that
\[\begin{split}
F_t^*\omega_X(y,z)&=e^{-t}\mn\sum_{i,j=1}^m g_{i\ov{j}}(ye^{-t/2},z) dy_i\wedge d\ov{y}_j\\
&+2e^{-t/2}\mathrm{Re}
\mn\sum_{i=1}^m\sum_{\alpha=1}^{n} g_{i\ov{\alpha}}(ye^{-t/2},z) dy_i\wedge d\ov{z}_\alpha\\
&+\mn\sum_{\alpha,\beta=1}^{n}g_{\alpha\ov{\beta}}(ye^{-t/2},z) dz_{\alpha}\wedge d\ov{z}_{\beta}.
\end{split}\]

Clearly as $t$ goes to infinity, these metrics converge smoothly on compact sets of $\mathbb{C}^m\times D$ to
the nonnegative form $\eta=\mn\sum_{\alpha,\beta=1}^{n}g_{\alpha\ov{\beta}}(0,z) dz_{\alpha}\wedge d\ov{z}_{\beta}$,
which is constant in the $y$ directions and is just equal to the restriction of $\omega_X$ on the fiber $X_{y_0}\cap U$,
under the identification $X_{y_0}\cap U=\{0\}\times D$.
The rescaled pullback metrics $e^tF_t^*\omega(t)$ are Ricci-flat K\"ahler on $B_t\times D$.
Next, we pull back equation \eqref{c2} and multiply it by $e^t$ and get
\begin{equation}\label{c0met}
C^{-1}\left(e^t F_t^*\chi+F_t^*\omega_X\right)\leq e^t F_t^*\omega\leq C\left(e^t F_t^*\chi+F_t^*\omega_X\right),
\end{equation}
which holds on $B_t\times D$. But $\chi$ is the pullback of a metric
$$\mn\sum_{i,j=1}^m \chi_{i\ov{j}}(y) dy_i\wedge d\ov{y}_j$$ on $B$ and so we have that
$$e^t (F_t^*\chi)(y,z)=\mn\sum_{i,j=1}^m \chi_{i\ov{j}}(ye^{-t/2}) dy_i\wedge d\ov{y}_j,$$
which as $t$ goes to infinity converge smoothly on compact sets of $\mathbb{C}^m\times D$ to
the nonnegative form $\eta'=\mn\sum_{i,j=1}^m \chi_{i\ov{j}}(0) dy_i\wedge d\ov{y}_j$. In particular, $\eta+\eta'$ is a K\"ahler metric on $\mathbb{C}^m\times D$
which is uniformly equivalent to $\omega_{E}$ on the whole of $\mathbb{C}^m\times D$. Since $\eta+\eta'$ is the limit
as $t$ goes to infinity of $e^t F_t^*\chi+F_t^*\omega_X$, we see that given any compact set $K$ of $B_t\times D$ there
is a constant $C$ independent of $t$ such that on $K$ we have
\begin{equation}\label{maa2}
C^{-1}\omega_{E}\leq e^t F_t^*\omega\leq C\omega_{E}.
\end{equation}
Using  \eqref{maa2}, we can then apply Lemma \ref{localest} to get that the Ricci-flat metrics $e^t F_t^*\omega$
have uniform $C^\infty$ bounds on any compact subset of $B_t\times D$.
If we now restrict to $\{0\}\times D$, which is identified with $X_{y_0}\cap U$, then the maps $F_t$ are just the identity, and by covering $X_{y_0}$ by finitely many such charts this shows
that the rescaled metrics along any smooth fiber $ e^t \omega|_{X_{y_0}}$ have uniform $C^\infty(X_{y_0},\omega_X|_{X_{y_0}})$ bounds. The uniform lower bound for $ e^t \omega|_{X_{y_0}}$
follows at once from \eqref{c2}.

Finally, the fact that the estimates are uniform as we vary $y_0$ in a compact subset of $B\backslash S'$ follows from the fact that all the constants in
the proof we just finished vary continuously as we vary $y_0$.
\end{proof}

Next, we consider the second setup from the Introduction, of collapsing of the K\"ahler-Ricci flow. In this case the same estimate as in \eqref{c2} was proved in \cite{FZ} (see also \cite{ST}), and the local estimates that replace Lemma \ref{localest} are given in \cite{SW}.
\begin{proof}[Proof of Theorem \ref{main1} for the K\"ahler-Ricci flow]
Given $x\in X\backslash S$, $y_0=\pi(x)$, and $t_k\to\infty$, we will show that there are constant $C_\ell, \ell\geq 0,$ such that
\begin{equation}\label{goall}
\left\| e^{t_k}\omega(t_k)|_{X_{y_0}}\right\|_{C^\ell(X_{y_0},\omega_X|_{X_{y_0}})}\leq C_\ell,\quad
e^{t_k}\omega(t_k)|_{X_{y_0}}\geq C_0^{-1} \omega_X|_{X_{y_0}},
\end{equation}
for all $k,\ell\geq 0$, and that these estimates are uniform as $y_0$ varies in a compact set of $B\backslash S'$. Once this is proved, the estimates stated in Theorem \ref{main1} follow easily from an argument by contradiction, using compactness.

As in the case of Ricci-flat K\"ahler metrics, we choose a chart $U$ centered at $x$ with local product coordinates $(y,z)$ as before, with $y$ in the unit polydisc in $\mathbb{C}^m$ and $z$ in the unit polydisc $D$ in $\mathbb{C}^n$, and define stretching maps
$$F_k:B_k\times D\to U, \quad F_k(y,z)=(ye^{-t_k/2},z),$$
where $B_k=B_{e^{t_k/2}}(0)\subset\mathbb{C}^m$. Thanks to the analog of Lemma \ref{c2} in \cite{FZ}, on $U$ we have
$$C^{-1}(\chi+e^{-t}\omega_X)\leq \omega(t)\leq C(\chi+e^{-t}\omega_X).$$
We consider the parabolically rescaled and stretched metrics
$$\ti{\omega}_k(t)=e^{t_k}F_k^* \omega(te^{-t_k} + t_k),\quad t\in[-1,0].$$
On $U$ we have that
$$C^{-1}(e^{t_k}\chi+e^{-te^{-t_k}}\omega_X)\leq e^{t_k} \omega(te^{-t_k} + t_k)\leq C(e^{t_k}\chi+e^{-te^{-t_k}}\omega_X),$$
and since we assume that $t\in [-1,0]$, we obtain
$$C^{-1}(e^{t_k}\chi+\omega_X)\leq e^{t_k} \omega(te^{-t_k} + t_k)\leq C(e^{t_k}\chi+\omega_X),$$
and so
\begin{equation}\label{c0met2}
C^{-1}F_k^*(e^{t_k}\chi+\omega_X)\leq \ti{\omega}_k(t)\leq CF_k^*(e^{t_k}\chi+\omega_X).
\end{equation}
The same calculation as in the case of Ricci-flat metrics shows that the metrics $F_k^*(e^{t_k}\chi+\omega_X)$ converge smoothly on compact
sets of $\mathbb{C}^m\times D$ to the product of a flat metric on $\mathbb{C}^m$ and the metric $\omega_X|_{X_{y_0}\cap D}$, and in particular given any compact subset
$K$ of $B_k\times D$ there is a constant $C$ such that on $K$ we have
$$C^{-1}\omega_E\leq F_k^*(e^{t_k}\chi+\omega_X)\leq C\omega_E,$$
where $\omega_E$ is the Euclidean metric on $\mathbb{C}^{n+m}$. Therefore on $K\times[-1,0]$ we have
$$C^{-1}\omega_E\leq \ti{\omega}_k(t)\leq C\omega_E.$$
The metrics $\ti{\omega}_k(t)$ for $t\in [-1,0]$ satisfy
\begin{equation}\label{rescaled2}
\frac{\de}{\de t}\ti{\omega}_k(t)=-\Ric(\ti{\omega}_k(t))-e^{-t_k}\ti{\omega}_k(t),
\end{equation}
and note that the coefficient $e^{-t_k}$ is uniformly bounded (and in fact goes to zero).
Then the interior estimates of \cite{SW} give us that, up to shrinking $K$ slightly,
$$\|\ti{\omega}_k(t)\|_{C^\ell(K,\omega_E)}\leq C_\ell,\quad \ti{\omega}_k(t)\geq C_0^{-1}\omega_E,$$
for $t\in [-1/2,0]$ and for all $k,\ell$, and for some uniform constants $C_\ell$. Setting $t=0$ we obtain local $C^\infty$ bounds for
the metrics $e^{t_k}F_k^* \omega(t_k)$.
If we now restrict to $\{0\}\times D$, which is identified with $X_{y_0}\cap U$, then the maps $F_k$ are just the identity, and by covering $X_{y_0}$ by finitely many such charts this proves \eqref{goall}.
Again, the fact that the estimates are uniform as $y_0$ varies in a compact subset of $B\backslash S'$ follows from the proof we just finished.
\end{proof}

\section{Blowup limits of the K\"ahler-Ricci flow}\label{blowup}

In this section we describe the possible blowup limits at time infinity of the K\"ahler-Ricci flow in the same setup as in the Introduction, thus proving Theorem \ref{main3}.

\begin{proof}[Proof of Theorem \ref{main3}]
Given a point $x\in X\backslash S$, $y_0=\pi(x)$, and $t_k\to\infty$, choose a chart $U$ centered at $x$ with local product coordinates $(y,z)$ as before, where $z$ varies in the unit polydisc $D\subset\mathbb{C}^n$ and $y$ in the unit polydisc in $\mathbb{C}^m$.
We let
$$\omega_k(t)=e^{t_k} \omega(te^{-t_k} + t_k),$$
be the parabolic rescalings of the metrics along the flow, with $t\in[-2,1]$. As in the proof of Theorem \ref{main1}, we can define stretchings
$$F_k:B_k\times D\to U, \quad F_k(y,z)=(ye^{-t_k/2},z),$$
where $B_k=B_{e^{t_k/2}}(0)\subset\mathbb{C}^m$, and we have that the metrics $F_k^*\omega_k(t), t\in[-1,0]$ have uniform $C^\infty$ bounds on compact sets of
$\mathbb{C}^m\times D$, and therefore up to passing to a subsequence they converge smoothly to a solution $\omega_\infty(t), t\in[-1,0]$ of the K\"ahler-Ricci flow
$$\frac{\de}{\de t}\omega_\infty(t)=-\Ric(\omega_\infty(t)),$$
on $\mathbb{C}^m\times D$. This limit flow is unnormalized because the coefficient $e^{-t_k}$ in \eqref{rescaled2} converges to zero. Also, passing to the limit in \eqref{c0met2} we see that
\begin{equation}\label{c0met3}
C^{-1}(\omega_X|_{X_{y_0}\cap D} + \omega_E)\leq \omega_\infty(t)\leq C (\omega_X|_{X_{y_0}\cap D} + \omega_E),
\end{equation}
on the whole of $\mathbb{C}^m\times D$ and for all $t\geq 0$, where $\omega_E$ is a flat metric on $\mathbb{C}^m$.
However, the original metrics $\omega(t)$ have a uniform bound on their scalar curvature thanks to \cite{ST3}, so after rescaling the scalar curvature
of $\ti{\omega}_k(t)$ goes to zero uniformly as $k$ goes to infinity. Therefore $\omega_\infty(t)$ is scalar flat, and from the pointwise evolution equation for its scalar curvature, we see that $\omega_\infty(t)$ is Ricci-flat, and hence independent of $t$.

Out next task is to glue together these local limits to obtain a global Cheeger-Gromov limit on $X_{y_0}\times\mathbb{C}^m$. To do this, we cover $X_{y_0}$ by finitely many charts $\{U_\alpha\}$ as above, and let $V=\cup_\alpha U_\alpha$. On each $U_\alpha$ we have local product coordinates $(y,z^\alpha)$, with $y$ and $z^\alpha$ belonging to the unit polydiscs in $\mathbb{C}^m$ and $\mathbb{C}^n$ respectively, and $\pi$ is given by $(y,z^\alpha)\mapsto y$. If $U_\alpha\cap U_\beta\neq\emptyset$, then on this overlap the two local coordinates are related by holomorphic transformation maps
$$z^\beta=h^{\alpha\beta}(y,z^\alpha).$$
Fix a radius $R>0$ and let $\ti{B}_k=B_{e^{-t_k/2}R}(0)$ be the polydisc of radius $e^{-t_k/2}R$, where $k$ is large enough so that $e^{-t_k/2}R<1$. We define a holomorphic coordinate change on $\ti{B}_k$ by
$y_k=e^{t_k/2}y,$ so that $\ti{B}_k$ becomes the polydisc $B_R(0)$ of radius $R$ in these new coordinates.
This is the same as applying the stretching maps $F_k$ as earlier. Then the complex manifold $V_k=\pi^{-1}(\ti{B}_k)$ (here we are viewing $\ti{B}_k$ as a subset of $B$) is obtained by gluing the polydiscs
$$\{(w,z^\alpha)\ |\ w\in B_R(0), z^\alpha\in D\},\quad \{(w,z^\beta)\ |\ w\in B_R(0), z^\beta\in D\}$$
via the transformation maps
$$z^\beta=h^{\alpha\beta}(we^{-t_k/2},z^\alpha).$$
As $k\to\infty$, these converge smoothly on compact sets to the transformation maps
$$z^\beta=h^{\alpha\beta}(0,z^\alpha),$$
which give the product manifold $B_R(0)\times X_{y_0}$.
Also, the metrics $e^{t_k}(\chi+e^{-{t_k}}\omega_X)$ on $V_k$ after changing coordinates converge smoothly on compact sets as $k\to\infty$ to the product metric $\omega_E+\omega_X|_{X_{y_0}}$ where $\omega_E$ is a Euclidean metric on $\mathbb{C}^m$, as in the proof of Theorem \ref{main1}.
On the other hand, as discussed above, the metrics $\omega_k(t)$ after coordinate change live on $V_k$ and after passing to a subsequence converge smoothly on compact sets to a Ricci-flat K\"ahler metric $\omega_\infty$ (independent of $t$) on $B_R(0)\times X_{y_0}$. By making $R$ larger and larger, we obtain the desired Cheeger-Gromov convergence of the flow on $V$ to a Ricci-flat K\"ahler metric $\omega_\infty$ on $X_{y_0}\times \mathbb{C}^m$. Also, thanks to \eqref{c0met3}, we have that $\omega_\infty$ is uniformly equivalent to $\omega_E+\omega_X|_{X_{y_0}}$ on $X_{y_0}\times\mathbb{C}^m$.

It remains to show that $\omega_\infty$ is the product of a Ricci-flat K\"ahler metric on $X_{y_0}$ and of a flat metric on $\mathbb{C}^m$. Thanks to the recent results in \cite[(1.8)]{TWY} we know that the restrictions $\omega_\infty|_{X_y}, y\in\mathbb{C}^m$, are all equal to the same Ricci-flat K\"ahler metric $\omega_F:=\omega_{SRF,y_0}$ on $X_{y_0}$ in the class $[\omega_X|_{X_{y_0}}]$. For simplicity of notation, let us also write $F=X_{y_0}$.

We claim that there is a smooth function $u$ on $F\times\mathbb{C}^m$ such that
\begin{equation}\label{ddb}
\omega_\infty=\omega_F+\omega_E+\ddbar u,
\end{equation}
on $F\times\mathbb{C}^m$, where $\omega_F+\omega_E$ is the product Ricci-flat metric.
To prove \eqref{ddb}, recall \cite{ST2} that the K\"ahler-Ricci flow is of the form
$$\omega(t)=(1-e^{-t})\chi+e^{-t}\omega_X+\ddbar\vp(t),$$
for some potentials $\vp(t)$. Therefore on $B_k\times D$ we can write
\begin{equation}\label{gg}
\ti{\omega}_k(t)=F_k^*\omega_k(t)=\omega_{\mathrm{ref},k}(t)+\ddbar(e^{t_k}F_k^*\vp(te^{-t_k}+t_k)),
\end{equation}
where as mentioned above the reference metrics
$$\omega_{\mathrm{ref},k}(t)=e^{t_k}F_k^*\left((1-e^{-te^{-t_k}-t_k})\chi+e^{-te^{-t_k}-t_k}\omega_X\right),$$
converge smoothly on compact sets to the product metric $\omega_E+\omega_X|_{F}$ (independent of $t$), as $k\to\infty$.

 We  observe that $[\omega_\infty]=[\omega_F+\omega_E]$ in $H^2(F\times\mathbb{C}^m)$, and so there is a real $1$-form $\zeta$ on $X\times\mathbb{C}^m$ such that
\begin{equation}\label{ddb1}
\omega_\infty=\omega_F+\omega_E+d\zeta=\omega_F+\omega_E+\de\zeta^{0,1}+\ov{\de\zeta^{0,1}}, \quad \db\zeta^{0,1}=0,
\end{equation}
where $\zeta=\zeta^{0,1}+\ov{\zeta^{0,1}}$. Since $c_1(F)=0$, the Bogomolov-Calabi decomposition theorem shows that there is a finite unramified cover $\pi:T\times \ti{F}\times \mathbb{C}^m\to F\times\mathbb{C}^m$, where $T$ is a torus (possibly a point), and $\ti{F}$ is simply connected with $c_1(\ti{F})=0$ (also possibly a point).
The Leray spectral sequence computing the Dolbeault cohomology of the product $T\times \ti{F}\times \mathbb{C}^m$ degenerates at the first page, giving
$$H^{0,1}(T\times \ti{F}\times \mathbb{C}^m)\cong H^{0,1}(T)\otimes H^{0}(\mathbb{C}^m,\mathcal{O}_{\mathbb{C}^m}),$$
where we used that $H^{0,1}(\mathbb{C}^m)=0$ by the $\db$-Poincar\'e Lemma.
If we write $T=\mathbb{C}^k/\Lambda$, and let $\{z_i\}$ be the standard coordinates on $\mathbb{C}^k$, then $H^{0,1}(T)$ is generated by the constant coefficient forms $\{d\ov{z}_i\}$, and so on $T\times \ti{F}\times \mathbb{C}^m$ we can write
$$\pi^*\zeta^{0,1}=\sum_{i=1}^k \sigma_i (y) d\ov{z}_i+\db h,$$
for some holomorphic functions  $\sigma_i (y)$ on $\mathbb{C}^m$ and a complex-valued function $h$,
where here $d\ov{z}_i$ also denote their pullbacks to $T\times \ti{F}\times \mathbb{C}^m$. Therefore $$ \pi^{*}\omega_\infty  = \pi^{*}(\omega_F+\omega_E)+2{\rm Re} \sum_{i}d\sigma_{i}\wedge d\bar{z}_{i} +2\sqrt{-1}\partial\overline{\partial}{\rm Im}h.  $$
Note  that $\omega_F$ is $\de\db$-cohomologous to $\omega_X|_F$ on $F$, and hence $\omega_F=\omega_X|_F+\sqrt{-1}\partial\overline{\partial} v$ for a smooth function $v$.  If we denote by $$u_{k}=\pi^{*}(e^{t_k}F_k^*\vp(te^{-t_k}+t_k)-v)-2{\rm Im}h,$$ then   by (\ref{gg}) we have that, $$\sqrt{-1}\partial\overline{\partial} u_{k}\rightarrow \pi^{*}\omega_\infty  - \pi^{*}(\omega_F+\omega_E)-2\sqrt{-1}\partial\overline{\partial}{\rm Im}h=2{\rm Re} \sum_{i}d\sigma_{i}\wedge d\bar{z}_{i},$$ smoothly on compact subsets of $T\times \ti{F}\times \mathbb{C}^m$. In particular, restricting this to any fiber $T\times\ti{F}\times \{y\}, y\in\mathbb{C}^m,$ we see that $$\ddbar u_k|_{T\times\ti{F}\times \{y\}}\to 0,$$ smoothly as $k\to\infty$. If we let $\underline{u}_k$ be the smooth function on $\mathbb{C}^m$ obtained by averaging $u_k$ on these fibers, with respect to $\pi^*\omega_F^n$, then $\ddbar \underline{u}_k$ also has uniform local $C^\infty$ bounds. The functions $u_k-\underline{u}_k$ thus have fiberwise average zero and the forms $\ddbar(u_k-\underline{u}_k)$ have uniform $C^\infty$ bounds on compact sets and their fiberwise restrictions go to zero smoothly, and therefore $u_k-\underline{u}_k$ converge to zero locally uniformly, and hence locally smoothly on $T\times \ti{F}\times \mathbb{C}^m$. It follows that $\ddbar(u_k-\underline{u}_k)\to 0$ locally smoothly, and so the form $2{\rm Re} \sum_{i}d\sigma_{i}\wedge d\bar{z}_{i},$ which is the limit of $\ddbar u_k$, is also equal to the limit of $\ddbar \underline{u}_k$. But since $\ddbar \underline{u}_k$ are forms on $\mathbb{C}^m$, this implies that $2{\rm Re} \sum_{i}d\sigma_{i}\wedge d\bar{z}_{i}$ is also the pullback of a form on $\mathbb{C}^m$, which is only possible if $d\sigma_i=0$ for all $i$. This gives $$\pi^*\de\zeta^{0,1}=\de\db h.$$
If we let $\ti{h}$ be the average of $h$ under the action of the (finite) deck transformation group of $\pi$, then $\ti{h}$ descends to a complex-valued function on $F\times\mathbb{C}^m$ and we have
$$\de\zeta^{0,1}=\de\db \ti{h},$$
which together with \eqref{ddb1} proves \eqref{ddb} with $u=2\mathrm{Im}\ti{h}$. This argument is similar to \cite[Proposition 3.1]{GTZ}, but we have exploited the more special setup here to obtain \eqref{ddb} without the need of passing to a holomorphic ``translation'' in the $T$ factor, which is needed in general.

Now that we have \eqref{ddb}, we can restrict it to any fiber $F\times\{y\}$, and since we have that $\omega_{\infty}|_{F\times\{y\}}=\omega_F$ for all $y$, we conclude that
$(\ddbar u)|_{F\times\{y\}}=0$, i.e. $u|_{F\times\{y\}}$ is a constant (which depends on $y$). In other words, $u$ is the pullback of a smooth function on $\mathbb{C}^m$. Therefore \eqref{ddb} now says that $\omega_\infty$ is the product of the Ricci-flat K\"ahler metric $\omega_F$ on $F$ times the K\"ahler metric $\hat{\omega}:=\omega_E+\ddbar u$ on $\mathbb{C}^m$. Since $\omega_\infty$ is Ricci-flat, we have that $\hat{\omega}$ is Ricci-flat as well, and the only thing left to prove is that $\hat{\omega}$ is flat. But earlier we proved that
$$C^{-1}(\omega_F+\omega_E)\leq \omega_\infty\leq C(\omega_F+\omega_E),$$
on the whole of $F\times\mathbb{C}^m$, and so we conclude that
$$C^{-1}\omega_E\leq \hat{\omega}\leq C\omega_E,$$
on $\mathbb{C}^m$.
We can write $\hat{\omega}=\ddbar \vp$ on $\mathbb{C}^m$ and the function $\log\det(\vp_{i\ov{j}})$ is harmonic and bounded on $\mathbb{C}^m$, hence constant, and we can then apply \cite[Theorem 2]{RS} to conclude that $\hat{\omega}$ is flat.
\end{proof}

In the end of the proof we used a Liouville-type theorem that a Ricci-flat K\"ahler metric on $\mathbb{C}^m$ which is uniformly equivalent to the Euclidean metric must be flat. The proof in \cite{RS} uses integral estimates, but in fact this result can also be easily proved using a local Calabi-type $C^3$ estimate as in \cite{Ya} (we leave the details to the interested reader).
Another related Liouville theorem was proved recently in \cite{Wa}.

Lastly, we mention the analogous result as Theorem \ref{main3} for the case of Ricci-flat K\"ahler metrics (the first setup in the Introduction).
\begin{theorem}
Assume the first setup in the Introduction, as in \eqref{ma}.
Given $x \in X\backslash S$, let $V$ be the preimage of a sufficiently small neighborhood of $\pi(x)$.
Then $(V, e^t\omega(t), x),$ converge in the smooth Cheeger-Gromov sense as $t\to\infty$ to $(X_y\times\mathbb{C}^m,\omega_\infty,(x,0))$, where $y=\pi(x)$,
and $\omega_\infty$ is the product of the unique Ricci-flat K\"ahler metric on $X_y$ in the class $[\omega_X|_{X_y}]$ and of a flat metric on $\mathbb{C}^m$.
\end{theorem}
\begin{proof}
Since the proof is almost identical to the one of Theorem \ref{main3}, we will be very brief. First we work on a small polydisc centered at $x$ with local product coordinates,
and as in the proof of Theorem \ref{main1} we show that after stretching the coordinates the metrics $e^t\omega(t)$ have uniform $C^\infty$ bounds, and therefore we obtain sequential limits
which are Ricci-flat K\"ahler metrics. As in the proof of Theorem \ref{main3} limits on different charts glue together to give a Ricci-flat K\"ahler metric on $X_y\times\mathbb{C}^m$, which is uniformly equivalent to a product metric. The same argument as before shows that such a metric is unique up to holomorphic isometry, and is the product of a Ricci-flat metric on $X_y$ and a flat metric on $\mathbb{C}^m$. It follows that this is the Cheeger-Gromov limit of the whole family $e^t\omega(t)$, without passing to subsequences.
\end{proof}

\section{Infinite time singularities of the K\"ahler-Ricci flow}\label{singular}

In this section we study the singularity types of long-time solutions of the K\"ahler-Ricci flow.

We wish to prove the criterion stated in Proposition \ref{crit}, which ensures that a long-time solution of the unnormalized K\"ahler-Ricci flow \eqref{krf} is of type IIb. First, we make an elementary observation: if there are points $x_k\in X$, times $t_k\to\infty$, $2$-planes $\pi_k\subset T_{x_k}X$ and a constant $\kappa>0$ such that
\begin{equation}\label{curv}
\mathrm{Sec}_{\omega(t_k)}(\pi_k)\geq \kappa,
\end{equation}
for all $k$, then in particular $\sup_X|\mathrm{Rm}(\omega(t_k))|_{\omega(t_k)}\geq \kappa,$ and so
$$\sup_{X\times [0,\infty)} t|\mathrm{Rm}(\omega(t))|_{\omega(t)}=+\infty,$$
and the flow is type IIb.

We now consider Proposition \ref{crit}. To give an intuition for why such a result should hold, let us first assume that the map $f:\mathbb{CP}^1\to X$ is a smooth embedding. Then we can apply the Gau\ss-Bonnet theorem to get
$$4\pi =\int_{\mathbb{CP}^1} K(f^*\omega(t)) f^*\omega(t)\leq \sup_{\mathbb{CP}^1}K(f^*\omega(t))\int_C \omega(t)\leq \sup_{C} K(\omega(t)) \int_C\omega_X,$$
where $K(f^*\omega(t))$ denotes the Gau\ss\ curvature of the pullback metric, while $\sup_C K(\omega(t))$ is the maximum of the bisectional curvatures of $\omega(t)$ at points of $C\subset X$. Here we used the facts that the bisectional curvature decreases in submanifolds, and that $\int_C c_1(X)=0$. We can therefore apply the observation above and conclude that the flow is type IIb. We now give the proof in the general case.
\begin{proof}[Proof of Proposition \ref{crit}]
Let $\omega(t)$ be any solution of the unnormalized K\"ahler-Ricci flow \eqref{krf}. Since $K_X$ is nef, we know that $\omega(t)$ exists for all positive $t$.
Note that the existence of a rational curve $C$ in $X$ implies that $X$ does not admit any K\"ahler metric with nonpositive bisectional curvature, by Yau's Schwarz Lemma \cite{Ya2}.
In particular $K(t)=\sup_X \mathrm{Bisec}_{\omega(t)},$ the largest bisectional curvature of $\omega(t),$ satisfies $K(t)>0$ for all $t$. Our goal is to give an effective uniform positive lower bound for $K(t)$.
We now work on $\mathbb{C}\subset\mathbb{CP}^1$, so that we have a nonconstant entire holomorphic map $f:\mathbb{C}\to X$. Consider the time-dependent smooth function $e(t)$ on $\mathbb{C}$ (the ``energy density'') given by
$$e(t)=\tr{\omega_E}{(f^*\omega(t))},$$
where $\omega_E$ is the Euclidean metric on $\mathbb{C}$. The usual Schwarz Lemma calculation \cite{Ya2} gives
$$\Delta_{E} e(t) \geq -K(t) e(t)^2.$$
Since the map $f$ is nonconstant, there is one point in $\mathbb{C}$ where $e(t)$ is nonzero for all $t$, and we may assume that this is the origin. A standard $\ve$-regularity argument (see e.g. \cite[Lemma 4.3.2]{MS}) shows that for each fixed $t$ if
\begin{equation}\label{uno}
\int_{B_r}e(t)\leq \frac{\pi}{8K(t)},
\end{equation}
then we have
\begin{equation}\label{due}
e(t)(0)\leq \frac{8}{\pi r^2}\int_{B_r}e(t),
\end{equation}
where $B_r$ is the Euclidean disc of radius $r$ centered at the origin. Thanks to the assumption that $\int_Cc_1(X)=0$, we have
$$\lim_{r\to\infty}\int_{B_r}e(t)=\int_C\omega(t)=\int_C \omega_X,$$
which is a positive constant. For a fixed $t$, if \eqref{uno} was true for all $r> 0$ then we could let $r\to\infty$ in \eqref{due} and obtain $e(t)(0)=0$, a contradiction. Therefore, for each $t$ there is some $r(t)>0$ with
$$\frac{\pi}{8K(t)}\leq \int_{B_{r(t)}}e(t)\leq \int_C\omega_X,$$
and so $K(t)\geq \kappa>0$. Since each bisectional curvature is the sum of two sectional curvatures, we conclude that $\omega(t)$ has some sectional curvature which is larger than $\kappa/2$, and we can thus apply the observation in \eqref{curv} to conclude that the flow is type IIb.
\end{proof}

\begin{remark} The argument above shows the following general conclusion about bisectional curvature.  If $(X, \omega)$ is a compact K\"{a}hler manifold containing a rational curve $C$, then $$ \sup_X \mathrm{Bisec}_\omega \geq \frac{\pi}{8\int_{C}\omega}.$$ In many special cases we know the existence of rational curves, for example $K3$ surfaces and quintic Calabi-Yau $3$-folds. More generally, the existence of rational curves  plays an  important role in the classification of projective varieties, which is one of the ingredients in the proof of Theorem \ref{main4}, Theorem \ref{surfaces} and Corollary \ref{corollary}. Another method  to obtain rational curves is to use the Gromov-Witten invariant. For a class $A\in H_{2}(X, \mathbb{Z} )$, if the Gromov-Witten invariant $GW_{A,\omega,0}^{X}$ of genus $0$ is defined, for instance when $X$ is a Calabi-Yau $3$-fold, and does not vanish (cf. \cite{MS}), then there is a rational curve $C\subset X$ with $\int_{C}\omega\leq \int_{A}\omega$ by the following argument. Let $J_{k}$ be a sequence of regular almost-complex structures compatible with $\omega$ and which converge to the complex structure $J$ of $X$. The assumption that $GW_{A,\omega,0}^{X}\neq 0$  implies that for every $k$ there is a $J_{k}$-holomorphic curve $f_{k}: \mathbb{CP}^{1} \rightarrow X$ representing $A$. If $f_{k}$ converges to a limit when $k\rightarrow\infty$, then we obtain a rational curve in $X$, and if not, we still have one by the bubbling process (see \cite{MS}). Because of the importance of Ricci-flat K\"{a}hler metrics  on Calabi-Yau $3$-folds, this lower bound of the supremum of the bisectional curvatures may have some independent interest.
\end{remark}

\begin{proof}[Proof of Theorem \ref{main4}]

{\bf Assume }$\kappa(X)=0$.

Since $K_X$ is semiample and $\kappa(X)=0$, we conclude that $K_X^\ell$ is trivial for some $\ell\geq 1$, i.e. $X$ is Calabi-Yau.

Case 1: $X$ is not a finite quotient of a torus, which is equivalent to the fact that any Ricci-flat K\"ahler metric constructed by Yau \cite{Ya} is not flat. Let $\omega_\infty$ be the unique Ricci-flat K\"ahler metric in the class $[\omega_X]$, and let $x\in X$ be a point with a $2$-plane $\pi\subset T_xX$ with $\mathrm{Sec}_{\omega_\infty}(\pi)\geq\kappa>0$ for some constant $\kappa>0$. Indeed we may choose $x$ to be any point where $\omega_\infty$ is not flat, and since the Ricci curvature vanishes, there must be a positive sectional curvature at $x$. Thanks to \cite{Ca}, we know that the solution $\omega(t)$ of the unnormalized flow \eqref{krf} converges smoothly to $\omega_\infty$ as $t\to\infty$. In particular $\mathrm{Sec}_{\omega(t)}(\pi)\geq\kappa/2>0$ for all $t$ large. Thanks to the observation in \eqref{curv}, the flow is type IIb.

Case 2: $X$ is a finite quotient of a torus. In this case the unnormalized flow \eqref{krf} converges smoothly to a flat metric $\omega_\infty$ \cite{Ca}. In fact, it is easy to see that this convergence is exponentially fast (cf. \cite{PS}). Briefly, one considers the Mabuchi energy
$$Y(t)=\int_X |\nabla \dot{\vp}|^2_{g(t)} \omega(t)^n,$$
where $\omega(t)=\omega_X+\ddbar\vp(t)$,
and using the smooth convergence of $\omega(t)$ to $\omega_\infty$ one easily shows that $dY/dt \leq -\eta Y$ for some $\eta>0$ and all $t\geq 0$, hence $Y\leq Ce^{-\eta t}.$ Similarly, for $k\geq 2$ one considers
$$Y_k(t)=\int_X |\nabla^k_{\mathbb{R}}\ \dot{\vp}|^2_{g(t)} \omega(t)^n,$$
and using interpolation inequalities and induction on $k$ one proves easily that $dY_k/dt\leq -2\eta Y_k+C_ke^{-\eta t},$ hence $Y_k\leq C_ke^{-\eta t}$. Using the Poincar\'e and Sobolev-Morrey inequalities, one deduces that $\|\dot{\vp}\|_{C^k(X,\omega_X)}\leq C_ke^{-\eta t}$, and the exponential smooth convergence of $\omega(t)$ to $\omega_\infty$ follows immediately.
Therefore there are constant $C,\eta>0$ such that $|\mathrm{Rm}(\omega(t))|_{\omega(t)}\leq Ce^{-\eta t}$, and it follows that the flow is type III.

{\bf Assume } $\kappa(X)=n$.

Case 1: $K_X$ is ample. In this case \cite{Ca} showed that the normalized K\"ahler-Ricci flow \eqref{nkrf} converges smoothly to the unique K\"ahler-Einstein metric $\omega_\infty$ with $\Ric(\omega_{\infty})=-\omega_\infty$. In particular, the sectional curvatures along the normalized flow remain uniformly bounded for all positive time. Translating this back to the unnormalized flow, we see that the flow is type III.

Case 2: $K_X$ is not ample. By assumption we have that $K_X$ is big and semiample (so in particular nef). It follows that $X$ is Moishezon and K\"ahler, hence projective.
Take $\ell$ sufficiently large and divisible so that sections of $K_X^\ell$ give a holomorphic map $f:X\to\mathbb{CP}^N$ with image a normal projective variety with at worst canonical singularities and ample canonical divisor. Since $K_X$ is not ample, $f$ is not an isomorphism with its image, and by Zariski's main theorem there is a fiber $F\subset X$ of $f$ which is positive dimensional. Each irreducible component of $F$ is uniruled by a result of Kawamata \cite[Theorem 2]{Ka}, and if $C$ is a rational curve contained in $F$ then $f(C)$ is a point and so $\int_C c_1(X)=0$. The criterion in Proposition \ref{crit} then shows that the flow is type IIb.

{\bf Assume } $0<\kappa(X)<n$.

Case 1: the generic fiber $X_y$ of the Iitaka fibration $\pi:X\to B$ is not a finite quotient of a torus.
Let $y\in B\backslash S'$ and fix a point $x\in X_y=\pi^{-1}(y)$. If $\omega(t)$ is the solution of the normalized K\"ahler-Ricci flow \eqref{nkrf}, and $\ti{\omega}(s)$ is the solution of the unnormalized flow \eqref{krf}, then we have that
$$\ti{\omega}(s)=e^t\omega(t), \quad s=e^t-1.$$
We have proved in Theorem \ref{main3} that if $U\supset X_y$ is the preimage of a sufficiently small neighborhood of $y$, then
there is a sequence $t_k\to\infty$ such that $(U,e^{t_k}\omega(t_k),x)$ converge smoothly in the sense of Cheeger-Gromov to $(X_y\times\mathbb{C}^m,\omega_\infty, (x,0))$, and $\omega_\infty$ is the product of a Ricci-flat K\"ahler metric on $X_y$ and a flat metric on $\mathbb{C}^m$. Since $X_y$ is not a finite quotient of a torus, we conclude that $\omega_\infty$ is not flat, and so
there is some point $x'\in X_y$ and a $2$-plane $\pi\subset T_{x'}(X_y\times\{0\})$ with $\mathrm{Sec}_{\omega_\infty}(\pi)\geq\kappa>0$ for some constant $\kappa>0$. Because of the smooth Cheeger-Gromov convergence, we conclude that (up to renaming the sequence $t_k$) there are $2$-planes $\pi_k\subset T_{x'}X$ with
$\mathrm{Sec}_{\ti{\omega}(t_k)}(\pi_k)\geq\kappa/2.$ By the observation in \eqref{curv}, the flow is type IIb.

Case 2: the generic fiber $X_y$ is a finite quotient of a torus, and $S=\emptyset$. We assume first that the Iitaka fibration $\pi:X\to B$ is a smooth submersion with fibers complex tori (not just finite quotients of tori). In this case, if $X$ is furthermore projective and $[\omega_X]$ is rational, Fong-Zhang \cite{FZ} adapted the estimates of Gross-Tosatti-Zhang \cite{GTZ} to prove that the solution $\omega(t)$ of the normalized K\"ahler-Ricci flow \eqref{nkrf} has uniformly bounded curvature for all time. The projectivity and rationality assumptions were recently removed in \cite{HT}.
Translating this back to the unnormalized flow, we see that the flow is type III.

Next, we treat the general case when the fibers are finite quotients of tori, and still $S=\emptyset$. Since $\pi$ is a proper submersion, it is a smooth fiber bundle so given any $y_0\in B$ we can find a small coordinate ball $U\ni y_0$ such that there is a diffeomorphism $f:U\times F\to\pi^{-1}(U)$ compatible with the projections to $U$, with $F$ diffeomorphic to a finite quotient of a torus. We pull back the complex structure from $\pi^{-1}(U)$, so we obtain a (in general non-product) complex structure on $U\times F$ which makes the map $f$ biholomorphic.
Let $\ti{F}\to F$ be a finite unramified covering with $\ti{F}$ diffeomorphic to a torus, and put the pullback complex structure on $U\times \ti{F}$ (again, in general not the product complex structure) so that the projection $p:U\times \ti{F}\to U\times F$ is holomorphic. The projection $\pi_U:U\times \ti{F}\to U$ equals $\pi\circ f\circ p$, and so is holomorphic, and therefore every fiber $\pi_U^{-1}(y),y\in U$ is a compact K\"ahler manifold diffeomorphic to a torus, hence it is biholomorphic to a torus (see e.g. \cite[Proposition 2.9]{Cat}). Therefore $f\circ p:U\times\ti{F}\to\pi^{-1}(U)$ is a holomorphic finite unramified covering, compatible with the projections to $U$, and $\pi_U:U\times \ti{F}\to U$ is a holomorphic submersion with fibers biholomorphic to complex tori,
and with total space K\"ahler. We can then use the local estimates in \cite{HT, FZ, GTZ} to conclude that the pullback of the normalized flow to $U\times\ti{F}$ has bounded curvature. But the metrics along this flow are all invariant under the group of holomorphic deck transformations of $f\circ p$, and therefore descend to the flow on $\pi^{-1}(U)$, which has also bounded curvature. Since $y_0$ was arbitrary, we conclude that the flow on $X$ is of type III.
\end{proof}

Lastly, we give the proof of Theorem \ref{surfaces}.

\begin{proof}[Proof of Theorem \ref{surfaces}]
Let $X$ be a minimal K\"ahler surface with $\kappa(X)=1$ and $\pi:X\to B$ an elliptic fibration which is not a submersion everywhere.
Thanks to Kodaira's classification of the singular fibers of elliptic surfaces \cite[V.7]{bhpv}, we see that either some singular fiber of $\pi$ contains a rational curve $C$ or otherwise the only singular fibers are of type $mI_0,m>1$, i.e. smooth elliptic curves with nontrivial multiplicity.

In the first case Kodaira's canonical bundle formula \cite{bhpv} gives
\begin{equation}\label{canbun}
K_X=\pi^*(K_B\otimes L)\otimes \mathcal{O}\left(\sum_i (m_i-1)F_i\right),
\end{equation}
for some line bundle $L$ on $B$, where $m_i$ is the multiplicity of the component $F_i$, i.e. $\pi^*(P_i)=m_i F_i,$ as Weil divisors, where $S'=\{P_i\}\subset B$ are all the critical values of $\pi$. Then if $\ell$ is sufficiently large so that $\ell(1-1/m_i)\in\mathbb{N}$ for all $i$, then we have
$$K_X^\ell=\pi^*\left((K_B\otimes L)^\ell\otimes\mathcal{O}\left(\sum_i \frac{\ell(m_i-1)}{m_i}P_i\right)\right),$$
and since $\pi(C)$ is a point, this shows that $\int_C c_1(X)=0$, and so we conclude that the flow is type IIb by Proposition \ref{crit},

It remains to show that if the only singular fibers are multiples of a smooth elliptic curve, then the flow is of type III, i.e. that
along the normalized flow \eqref{nkrf} the curvature remains uniformly bounded for all time.

On compact sets away from the singular fibers, this is true thanks to \cite{FZ} (cf. \cite{GTZ, HT, Gi}). Let then $\Delta\subset B$ be the unit disc in some coordinate chart such that $S'\cap \Delta=\{y_0\}$, the center of the disc, so the fiber $X_{y_0}$ is of type $mI_0$ for some $m>1$ (which could depend on the point $y_0$), and let $U=\pi^{-1}(\Delta)$. Then, thanks to the local description of such singular fibers \cite[Proposition 1.6.2]{FM}, we have a commutative diagram
\begin{equation}
\begin{CD}
 \ti{U} @>{p}>>X\\
@V{\ti{\pi}}VV @VV{\pi}V \\
 \Delta @>{q}>>\Delta
\end{CD}
\end{equation}
where $q(z)=z^m$, the map $\ti{\pi}:\ti{U}\to\Delta$ is a holomorphic submersion with fibers elliptic curves, and the map $p:\ti{U}\to U$ is a holomorphic finite unramified covering.
If we can show that the pullback of the normalized flow to $\ti{U}$ has bounded curvature, then the same is true for the flow on $U$, and since $y_0\in S'$ was arbitrary, we would conclude that the flow on $X$ is of type III.

As before let $\chi$ be the restriction of $\frac{1}{\ell}\omega_{FS}$ to $\Delta\subset B\subset \mathbb{P}H^0(K_X^\ell)$. This is a smooth semipositive form on $\Delta$, although $\chi$ may not be positive definite at the center of the disc.
On $\Delta$ let $v=|z|^{2/m}-\psi$, where $\psi$ is a potential for $\chi$ on $\Delta$ and define $\omega_B=\chi+\ddbar v$, which is an orbifold flat K\"ahler metric on $\Delta$, so that $q^*\omega_B$ is the Euclidean metric on $\Delta$ (and so $q^*v$ is smooth). For simplicity we will also denote $\pi^*v$ by $v$ and $\pi^*\omega_B$ by $\omega_B$, so that with this notation we have that $p^*v$ is smooth on $\ti{U}$. Up to shrinking $\Delta$, we may also assume that $v$ is defined in a neighborhood of $\ov{\Delta}$.
Thanks to the estimate \eqref{c2} for the normalized flow, which was proved in \cite{FZ} (see also \cite{ST}), and since on $\de U$ the semipositive form $\chi$ is uniformly equivalent to $\pi^*\omega_B$, we conclude that there is a constant $C$ such that on $\de U$ we have
$$C^{-1}(\pi^*\omega_B+e^{-t}\omega_X)\leq \omega(t)\leq C(\pi^*\omega_B+e^{-t}\omega_X),$$
for all $t>0$. Therefore on $\de\ti{U}$ we get
\begin{equation}\label{c2new}
C^{-1}(\ti{\pi}^*q^*\omega_B+e^{-t}p^*\omega_X)\leq p^*\omega(t)\leq C(\ti{\pi}^*q^*\omega_B+e^{-t}p^*\omega_X).
\end{equation}
Our goal is to prove the same estimate on the whole of $\ti{U}$. We follow the same strategy as in \cite{FZ} (cf. \cite{To, ST}). Recall \cite{ST2} that the K\"ahler-Ricci flow is of the form
$$\omega(t)=(1-e^{-t})\chi+e^{-t}\omega_X+\ddbar\vp(t),$$
and the potentials $\vp$ satisfy a uniform $L^\infty$ bound $|\vp(t)|\leq C$ for all $t\geq 0$. If we let $\ti{\vp}(t)=\vp(t)-(1-e^{-t})v$, then $\ti{\vp}$ is still uniformly bounded and
$\omega(t)=(1-e^{-t})\omega_B+e^{-t}\omega_X+\ddbar\ti{\vp}(t)$ holds on $U\backslash S$. The function $\ti{\vp}$ may not be smooth on $U\cap S$, but after pulling back to $\ti{U}$ it becomes $p^*\ti{\vp}=p^*\vp-(1-e^{-t})\ti{\pi}^*q^*v$ which is smooth everywhere, and we have
$$p^*\omega(t)=(1-e^{-t})\ti{\pi}^*q^*\omega_B+e^{-t}p^*\omega_X+\ddbar p^*\ti{\vp}(t).$$
The parabolic Schwarz Lemma calculation (cf. \cite{ST, Ya2}) applied to the map $\ti{\pi}:(\ti{U},p^*\omega(t))\to (\Delta, q^*\omega_B)$ gives on $\ti{U}$
$$\left(\frac{\de}{\de t}-\Delta_{p^*\omega(t)}\right) (\tr{p^*\omega(t)}{(\ti{\pi}^*q^*\omega_B)}-2p^*\ti{\vp}(t))\leq - \tr{p^*\omega(t)}{(\ti{\pi}^*q^*\omega_B)}+4,$$
for $t$ large.
Since the quantity $\tr{p^*\omega(t)}{(\ti{\pi}^*q^*\omega_B)}-2p^*\ti{\vp}(t)$ is uniformly bounded on $\de\ti{U}$ thanks to \eqref{c2new} and the bound on $\ti{\vp}$, the maximum principle gives
\begin{equation}\label{schw}
\tr{p^*\omega(t)}{(\ti{\pi}^*q^*\omega_B)}\leq C,
\end{equation}
on $\ti{U}\times [0,\infty).$
Let now $y$ be any point in $\Delta$ and consider the fiber $\ti{X}_y=\ti{\pi}^{-1}(y)$, which is a smooth elliptic curve. Restricting to $\ti{X}_y$ and using $\ti{\pi}^*q^*\omega_B\leq Cp^*\omega(t)$, as in \cite[Corollary 5.2]{ST} or \cite[(3.9)]{To}, we easily see that
\begin{equation}\label{osc}
\mathrm{osc}_{\ti{X}_y}(e^tp^*\ti{\vp})\leq C,
\end{equation}
independent of $y\in\Delta$ and $t\geq 0$. We then let $\hat{\vp}_y(t)$ to be the average of $p^*\ti{\vp}(t)$ on $\ti{X}_y$ with respect to the volume form $p^*\omega_X|_{\ti{X}_y}$. This defines a smooth function on $\Delta$, uniformly bounded for all $t\geq 0$, and we denote its pullback to $\ti{U}$ by $\hat{\vp}(t)$. The bound \eqref{osc} gives
$\sup_{\ti{U}}|e^t (p^*\ti{\vp}-\hat{\vp})|\leq C$ independent of $t\geq 0$. Then a calculation as in \cite{FZ} (see also \cite{ST, To}), gives
$$\left(\frac{\de}{\de t}-\Delta_{p^*\omega(t)}\right) (\log \tr{p^*\omega(t)}{(e^{-t}p^*\omega_X)}-Ae^t (p^*\ti{\vp}-\hat{\vp}))\leq - \tr{p^*\omega(t)}{(p^*\omega_X)}+CAe^t,$$
if $A$ is large enough.
Since the quantity $\log \tr{p^*\omega(t)}{(e^{-t}p^*\omega_X)}-Ae^t (p^*\ti{\vp}-\hat{\vp})$ is uniformly bounded on $\de\ti{U}$ thanks to \eqref{c2new} and the bound on $e^t (p^*\ti{\vp}-\hat{\vp})$, the maximum principle gives
$\tr{p^*\omega(t)}{(e^{-t}p^*\omega_X)}\leq C$ on $\ti{U}\times [0,\infty)$. Adding this estimate to \eqref{schw} we conclude that
\begin{equation}\label{lower}
\ti{\pi}^*q^*\omega_B+e^{-t}p^*\omega_X\leq Cp^*\omega(t),
\end{equation}
on $\ti{U}\times [0,\infty).$ Now the main theorem of \cite{ST3} shows that $|\dot{\vp}|\leq C$ holds on $X\times [0,\infty)$, where $\dot{\vp}=\de \vp/\de t$. We write the flow on $X$ as the parabolic complex Monge-Amp\`ere equation
$$\frac{\de}{\de t}\vp=\log\frac{e^{t}((1-e^{-t})\chi+e^{-t}\omega_X+\ddbar\vp)^{2}}{\Omega}-\vp,$$
where $\Omega$ is a smooth volume form on $X$ with $\ddbar\log\Omega=\chi$. Pulling back to $\ti{U}$, we may write it as
$$\frac{\de}{\de t}p^*\ti{\vp}=\log\frac{e^{t}((1-e^{-t})\ti{\pi}^*q^*\omega_B+e^{-t}p^*\omega_X+\ddbar p^*\ti{\vp})^{2}}{p^*\Omega}-p^*\ti{\vp},$$
and still have $|\de p^*\ti{\vp}/\de t|\leq C$ (recall that $p^*\ti{\vp}$ is smooth on $\ti{U}$). Then we see that on $\ti{U}$ the metrics $p^*\omega(t)$ and $\ti{\pi}^*q^*\omega_B+e^{-t}p^*\omega_X$ have uniformly equivalent volume forms,
both comparable to $e^{-t}p^*\Omega$, and this together with \eqref{lower} finally proves that \eqref{c2new} holds on all of $\ti{U}\times[0,\infty)$.

Now that we have established \eqref{c2new}, we can use the same procedure as in \cite{HT, FZ, GTZ} to conclude that the pullback of the normalized flow to $\ti{U}$ has bounded curvature, and we are done.
\end{proof}

For example, if $X$ is a minimal properly elliptic K\"ahler surface such that the sections of $K_X$ already give rise to the Iitaka fibration $\pi:X\to B$, then it is easy to see using Kodaira's canonical bundle formula \eqref{canbun} that there can be no multiple fibers of $\pi$, so in this case if $S\neq\emptyset$ then the flow is always of type IIb. It is also easy to construct examples of minimal properly elliptic K\"ahler surfaces with $S\neq\emptyset$ and all singular fibers of type $mI_0$, in which case the flow is of type III. For example we can take a compact Riemann surface $\Sigma$ of genus $g\geq 2$ with an automorphism $f:\Sigma\to\Sigma$ of order $m$ with finitely many fixed points, then take the elliptic curve $E=\mathbb{C}/(\mathbb{Z}\oplus i\mathbb{Z})$, and consider the free $\mathbb{Z}_m$-action on $\Sigma\times E$ generated by $(x,z)\mapsto (f(x),z+1/m)$, $x\in\Sigma, z\in E$. Then $\pi:X=(\Sigma\times E)/\mathbb{Z}_m\to B=\Sigma/\mathbb{Z}_m$ is a minimal properly elliptic K\"ahler surface with singular fibers of type $mI_0$ above the fixed points of $f$, and $\pi$ is the Iitaka fibration of $X$.

\begin{proof}[Proof of Corollary \ref{corollary}]  It is shown in \cite{Ka2} that $X$ and $Y$ can be connected by a nontrivial sequence of finitely many flops. In particular, by Lemma 2 of \cite{Ka2}, there is a rational curve $C$ in $X$ such that $\int_C c_1(X)=0$.  The conclusion then follows from Proposition \ref{crit}.
\end{proof}


\begin{thebibliography}{99}
\bibitem{Ba} R.H. Bamler {\em Long-time analysis of 3 dimensional Ricci flow III}, preprint, arXiv:1310.4483.
\bibitem{bhpv} W.P. Barth, K. Hulek, C.A.M. Peters, A. Van de Ven {\em Compact complex surfaces. Second edition}, Springer, 2004.
\bibitem{Ca} H.-D. Cao {\em Deformation of K\"ahler metrics to K\"ahler-Einstein metrics on compact K\"ahler manifolds}, Invent. Math. {\bf 81}  (1985), no. 2, 359--372.
\bibitem{Cat} F. Catanese {\em Deformation in the large of some complex manifolds, I}, Ann. di Matematica {\bf 183} (2004), 261--289.
\bibitem{FZ}  F.T.-H. Fong, Z. Zhang \emph{The collapsing rate of the K\"ahler-Ricci flow with regular infinite time singularity},  J. Reine Angew. Math. {\bf 703} (2015), 95--113.
\bibitem{FM} R. Friedman, J.W. Morgan {\em Smooth four-manifolds and complex surfaces}, Ergebnisse der Mathematik und ihrer Grenzgebiete, 27. Springer-Verlag, Berlin, 1994.
\bibitem{Gi} M. Gill {\em Collapsing of products along the K\"ahler-Ricci flow}, Trans. Amer. Math. Soc. {\bf 366} (2014), no. 7, 3907--3924.
\bibitem{GTZ} M. Gross, V. Tosatti, Y. Zhang {\em Collapsing of abelian fibred Calabi-Yau manifolds}, Duke Math. J. {\bf 162} (2013), no. 3, 517--551.
\bibitem{GTZ2} M. Gross, V. Tosatti, Y. Zhang {\em Gromov-Hausdorff collapsing of Calabi-Yau manifolds},  Comm. Anal. Geom. {\bf 24} (2016), no. 1, 93--113.
\bibitem{Ha} R.S. Hamilton {\em The Ricci flow on surfaces}, in {\em Mathematics and general relativity (Santa Cruz, CA, 1986)}, 237--262, Contemp. Math., 71, Amer. Math. Soc., Providence, RI, 1988.
\bibitem{HT} H.-J. Hein, V. Tosatti {\em Remarks on the collapsing of torus fibered Calabi-Yau manifolds},  Bull. Lond. Math. Soc. {\bf 47} (2015), no. 6, 1021--1027.
\bibitem{Ka} Y. Kawamata {\em On the length of an extremal rational curve}, Invent. Math. {\bf 105} (1991), no. 3, 609--611.
\bibitem{Ka2} Y. Kawamata {\em Flops connect minimal models}, Publ. RIMS. Kyoto Univ. {\bf 44} (2008), 419--423.
\bibitem{Ko} K. Kodaira {\em Complex manifolds and deformation of complex structures}, Springer, 2005.
\bibitem{Lo} J. Lott {\em On the long-time behavior of type-$III$ Ricci flow solutions}, Math. Ann. (3) {\bf 339} (2007), 627--666.
\bibitem{MS} D. McDuff, D. Salamon {\em $J$-holomorphic curves and symplectic topology}, American Mathematical Society Colloquium Publications, 52. American Mathematical Society, Providence, RI, 2004.
\bibitem{PSS} D.H. Phong, N. \v{S}e\v{s}um, J. Sturm {\em Multiplier Ideal Sheaves and the K\"ahler-Ricci Flow},  Comm. Anal. Geom.  {\bf 15}  (2007),  no. 3, 613--632.
\bibitem{PS} D.H. Phong, J. Sturm {\em On stability and the convergence of the K\"ahler-Ricci flow}, J. Differential Geom. {\bf 72} (2006), no. 1, 149--168.
\bibitem{RS} D. Riebesehl, F. Schulz {\em A priori estimates and a Liouville theorem for complex Monge-Amp\`ere equations}, Math. Z. {\bf 186} (1984), 57--66.
\bibitem{SW} M. Sherman, B. Weinkove {\em Interior derivative estimates for the K\"ahler-Ricci flow}, Pacific J. Math. {\bf 257} (2012), no. 2, 491--501.
\bibitem{ST} J. Song, G. Tian {\em The K\"ahler-Ricci flow on surfaces of positive Kodaira dimension}, Invent. Math. {\bf 170} (2007), no. 3, 609--653.
\bibitem{ST2} J. Song, G. Tian {\em Canonical measures and K\"ahler-Ricci flow}, J. Amer. Math. Soc. {\bf 25} (2012), no. 2, 303--353.
\bibitem{ST3} J. Song, G. Tian {\em Bounding scalar curvature for global solutions of the K\"ahler-Ricci flow},  Amer. J. Math. {\bf 138} (2016), no. 3, 683--695.
\bibitem{TZ} G. Tian, Z. Zhang {\em On the K\"ahler-Ricci flow on projective manifolds of general type}, Chinese Ann. Math. Ser. B {\bf 27} (2006), no. 2, 179--192.
\bibitem{To} V. Tosatti, {\em  Adiabatic limits of Ricci-flat K\"ahler metrics}, J. Differential Geom. {\bf 84} (2010), no. 2, 427--453.
\bibitem{To3} V. Tosatti {\em Degenerations of Calabi-Yau metrics}, in {\em Geometry and Physics in Cracow,} Acta Phys. Polon. B Proc. Suppl. {\bf 4} (2011), no.3, 495--505.
\bibitem{To4} V. Tosatti {\em Calabi-Yau manifolds and their degenerations}, Ann. N.Y. Acad. Sci. {\bf 1260} (2012), 8--13.
\bibitem{TWY} V. Tosatti, B. Weinkove, X. Yang {\em The K\"ahler-Ricci flow, Ricci-flat metrics and collapsing limits}, preprint, arXiv:1408.0161.
\bibitem{Wa} Y. Wang {\em A Liouville theorem for the complex Monge-Amp\`ere equation}, preprint, arXiv:1303.2403.
\bibitem{Ya}  S.-T. Yau  \emph{On the Ricci curvature of a compact K\"ahler manifold and the complex Monge-Amp\`ere equation, I}, Comm. Pure Appl. Math. {\bf 31} (1978), 339--411.
\bibitem{Ya2} S.-T. Yau {\em A general Schwarz lemma for K\"ahler manifolds}, Amer. J. Math. {\bf 100} (1978), no. 1, 197--203.
\bibitem{Z} Y. Zhang {\em Convergence of K\"ahler manifolds and calibrated fibrations}, PhD thesis, Nankai Institute of Mathematics, 2006.
\end{thebibliography}
\end{document}